\documentclass{llncs}
\usepackage{amsmath}
\usepackage{amsfonts}
\usepackage{tikz}
\usepackage{tikz-3dplot}

\newcommand{\set}[1]{\left\{#1\right\}}

\newcommand{\Real}{\mathbb{R}}

\newcommand{\To}{\rightarrow}

\newcommand{\cO}{\mathcal{O}}

\newcommand{\pic}{\pi^{\textrm{c}}}
\newcommand{\pir}{\pi^{\textrm{r}}}

\begin{document}

\tdplotsetmaincoords{60}{110}

\title{New symmetries in \\mixed-integer linear optimization}
\subtitle{Symmetry heuristics and complement-based symmetries}
\titlerunning{Novel symmetries in MILP} 
\author{Philipp M. Christophel\thanks{\email{philipp.christophel@sas.com}, corresponding author} \and Menal G\"uzelsoy \and Imre P\'olik}
\institute{SAS Institute, Inc.\\Cary, NC, USA}

\date{\today}%
\maketitle

\begin{abstract}
We present two novel applications of symmetries for mixed-integer linear programming. First we propose two variants of a new heuristic to improve the objective value of a feasible solution using symmetries. These heuristics can use either the actual permutations or the orbits of the variables to find better feasible solutions.

Then we introduce a new class of symmetries for binary MILP problems. Besides the usual permutation of variables, these symmetries can also take the complement of the binary variables. This is useful in situations when two opposite decisions are actually symmetric to each other. We discuss the theory of these symmetries and present a computational method to compute them.

Examples are presented to illustrate the usefulness of these techniques.
\end{abstract}
 
\section{Introduction}

\subsection{Symmetries in MILP}

Let us consider an integer programming problem\footnote{The concepts in this paper naturally extend to the mixed-integer linear case, but for simplicity of notation we will use the pure binary form.} of the form:
\begin{align}
 \notag \min\ &c^Tx\\
 \label{eq:MILPorig}Ax&\leq b\\
 \notag x&\in\set{0,1}^n,
\end{align}
where $x,c\in\Real^n$, $b\in\Real^m$ and $A\in\Real^{m\times n}$.
In general, following the concepts in \cite{ostrowski08,ostrowski11}, a \emph{symmetry} of problem \eqref{eq:MILPorig} is a permutation of its rows and columns such that the permuted problem is identical to the original one. In particular, let $\pir$ and $\pic$ be permutations of the rows and columns, respectively, then they form a symmetry of problem \eqref{eq:MILPorig} if
\begin{subequations}
\begin{align}
\label{eq:symcondmat}A_{i,j} &=A_{\pir(i),\pic(j)}\\
\label{eq:symcondrhs}b_j &= b_{\pir(j)}\\
\label{eq:symcondobj}c_i &= c_{\pic(i)}
\end{align}
\end{subequations}
Now, if $x$ is a feasible solution of \eqref{eq:MILPorig}, then $\pic(x)$, which is  defined naturally as 
\begin{equation}
\pic(x)_i = x_{\pic(i)}
\end{equation}
is also feasible with the same objective value, thus the two solutions are equivalent. It is important to detect the symmetry of the problem so that equivalent solutions are not appearing on the search tree. 

Another important concept related to symmetries is the notion of \emph{orbits}. Two columns (or rows) are said to belong to the same orbit if there is a permutation that maps one to the other. Naturally, each variable belongs to exactly one orbit, thus the orbits form a partition of the columns and rows. 

These techniques are well known (see, e.g., \cite{katebi12} for a recent overview) and have been implemented in all major commercial software packages, including SAS/OR 13.1.

Permutations will be presented with the usual cycle notation. Whenever we talk about symmetries of a problem we will usually drop the row permutations. The all-1 vector will be denoted by $e$.

The paper is structured as follows. In Section \ref{sec:symheur} we present two improvement heuristics that use symmetry information. In Section \ref{sec:compsym} we introduce the notion of complement symmetries and in Section \ref{sec:compsymalg} we present a computational method to obtain these symmetries.

\section{Symmetry heuristics}\label{sec:symheur}

In this section we are interested in symmetries of problem \eqref{eq:MILPorig} that satisfy conditions \eqref{eq:symcondmat} and \eqref{eq:symcondrhs}, but not \eqref{eq:symcondobj}. We will call these \emph{constraint symmetries}, since applying them to a solution maintains feasibility, but not the objective value. It is important that the group of constraint symmetries can be larger than the original symmetry group of the problem, and there are examples for this even in the public benchmark set MIPLIB2010 \cite{koch11} (see instances \texttt{bab5}, \texttt{msc98}, \texttt{neos\_1109824}, \texttt{ns1766074}, \texttt{ran16x16} and \texttt{unitcal\_7}, for example). 

We can use these permutations to improve the objective value of a feasible solution. There are multiple ways to use this information, here we present two different schemes and discuss which one is better in certain situations. 

\subsection{Direct improvement heuristics}

Having a list of generators of the group of constraint symmetries we can just apply them repeatedly to a feasible solution and see if we get a better objective value. An advantage of this heuristic is that using the symmetries guarantees that feasibility is preserved. In general, without the symmetry information this would be complicated to achieve. On the other hand, trying to permute the solution to minimize its objective value is equivalent to optimizing a linear objective function over a group defined by generators, a problem that is known to be NP hard \cite{buchheim05}. 

Also, the set of generators we have is usually very small. On one hand this is an advantage, because the algorithms that find these symmetries are faster, but it can easily happen that none of the symmetries improves the currently best solution. To alleviate this situation we use a solution pool of the best few solutions along with a symmetry pool. Whenever a new solution is found we put it in the pool and every once in a while we delete the solutions whose objective value is not good enough. On the other hand, if none of the symmetries in the symmetry pool can improve any of the solutions in the solution pool, then we can combine symmetries to increase the size of the symmetry pool. We iterate this a few times or as long as we are finding better solutions. 

This method is very fast, since the symmetries are usually permutations that move only a few elements, so checking whether they improve a solution is fast. Speed has its price, since this method cannot hope to cover all the solutions that can be generated. Another limitation of this method is that it cannot change the number of nonzeros in a solution. We will present an improvement in Section \ref{sec:compheur}.

\subsubsection{Example}

The following example illustrates the idea behind the direct improvement symmetry heuristic. Consider this problem:
\begin{align*}
\min \quad&   c^{T} x \\
                         Ax &\geq 1 \\
                  	        x &\in\set{0,1},
\end{align*}
where 
\begin{align*}
c &= \left(1\ 2\ 3\ 3\ 1\ 2\right)\\
A &= \left(
\begin{array}{cccccc}
1 & 0 & 1 & 0 & 0 & 0\\
0 & 1 & 0 & 1 & 0 & 0\\
1 & 0 & 0 & 0 & 1 & 0\\
0 & 1 & 0 & 0 & 0 & 1\\
0 & 0 & 1 & 0 & 1 & 0\\
0 & 0 & 0 & 1 & 0 & 1\\
\end{array}\right).
\end{align*}
A feasible solution is 
\begin{equation}
x = \left(1\ 1\ 1\ 1\ 0\ 0\right)
\end{equation}
with objective value 9. The constraint symmetries of this problem are generated by
\begin{equation}
\mathcal{G^C}(A) = \set{(1,3), (2,4), (4,6), (1,2)(3,4)(5,6)}.
\end{equation}
The following table shows how symmetries are used to improve the solution:
\begin{center}
\begin{tabular}{lccccc}
Nr & x & & z & & symmetry\\ \hline
1 & (1 1 1 1 0 0) & & 9 & &(4,6)\\
2 & (1 1 1 0 0 1) & & 8 & &(1,2)(3,4)(5,6)\\
3 & (1 1 0 1 1 0) & & 7 & &(4,6)\\
4 & (1 1 0 0 1 1) & & 6 & & optimal
\end{tabular}
\end{center}
There are a few key points to note here. First, because of the different objective coefficients, this problem has no symmetries, but it still has a nontrivial constraint symmetry group. This is actually typical for a lot of optimization problems, where the constraint set is symmetric and the objective coefficients are breaking the symmetry.

Further, it is important to note that not all permutations of 4 ones give a feasible solution, for example, $(0\ 1\ 0\ 1\ 1\ 1)$ is not feasible. This is because the permutation $(1,4)$ is not a constraint symmetry. The strength of our heuristic is that applying the symmetries we do not have to worry about feasibility. In fact, each solution we get this way will be exactly the same feasible as the original solution. Finally, this procedure is not guaranteed to yield an optimal solution (as it does in our example), it can get stuck in a locally optimal solution. This motivates our second approach.

\subsection{Orbit MIPping}

Another way to use symmetry information to improve the objective value of a feasible solution is to use only the orbit information. Variables that can be mapped to each other with symmetries are said to belong to the same orbit. The orbits give a partition of the variables of problem \eqref{eq:MILPorig} and it is easy to see that applying a symmetry to a solution cannot change the number of nonzeros on a given orbit. Our orbit mipping heuristic uses this simple fact to improve the solutions. Given a feasible solution $\tilde x$ and an orbit partition $\cO_1, \dots, \cO_k$ we can look at the following subproblem:
\begin{align}
 \notag \min\ &c^Tx\\
 \label{eq:MILPsym}Ax&\leq b\\
\notag \sum_{i\in\cO_j}x_i & = \sum_{i\in\cO_j}\tilde x_i,\, \forall j = 1,\dots,k\\
 \notag x&\in\set{0,1}^n.
\end{align}
In other words, for each orbit we fix the sum of variables to be the same as in $\tilde x$ and solve the resulting MILP. This will actually search a broader set than just the solutions that are symmetric to $\tilde x$. We can feed in $\tilde x$ as a feasible solution to the optimization when solving problem \eqref{eq:MILPsym}.

The rationale behind the efficiency of this method is that the extra equality constraints might imply a lot of fixings throughout the problem. This approach is useful if the constraint symmetry orbits of the solutions are large and there are only a few nonzeros on each orbit, or if the group of constraint symmetries is too large for the direct improvement heuristic to be efficient. A disadvantage of this method is that it is much slower than using the symmetries directly, but it can find better solutions. Orbit MIPping has a very important property that distinguish it from other improvement heuristics. If the set of constraint symmetries is larger than the symmetry group of the problem (something which is easy to check), then we know that there is another solution symmetric to the current solution, but with a different (better or worse) objective value.

\section{Complement-based symmetries}\label{sec:compsym}

In the second part of the paper we are presenting a new class of symmetries for MILP, we show how to compute them and bring examples from public test libraries.

\subsection{Motivation}

Consider the following MILP:
\begin{align}
 \notag \min\ x_1-x_2&\\
 \label{eq:2dexample} x_2&\leq x_1\\
 \notag x_1,x_2\ &\text{binary}
\end{align}
The feasible set is depicted in Figure \ref{fig:2dexample}. 
\begin{figure}[ht]
\begin{center}
\begin{tikzpicture}[domain=-0.2:1.2,x=3.2cm,y=3.2cm]
\draw[->] (0,-0.1) -- (0,1.1); 
\draw[->] (-0.1,0) -- (1.1,0); 
\draw[thick, fill = gray, fill opacity = 0.5] (0,0) -- (1,1) -- (1,0) -- cycle; 
\draw[->] (0.6,0.4) -- (0.75,0.25); 
\draw (0,0) node[above left] {$(0,0)$}; 
\draw (1,0) node[above right] {$(1,0)$}; 
\draw (1,1) node[above right] {$(1,1)$}; 
\draw[dashed] (-0.1,1.1) -- (1.1,-0.1); 
\end{tikzpicture}
\end{center}
\caption{The feasible set (shaded area) of the linear relaxation of problem \eqref{eq:2dexample}. The dashed line denotes the axis of symmetry. The arrow indicates the direction of the improving objective value.}\label{fig:2dexample}
\end{figure}
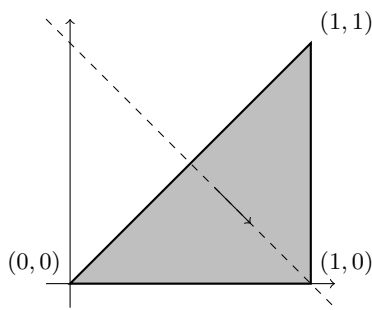
Notice that swapping the two variables (which is the only nontrivial permutation-based symmetry for this problem) does not map the feasible set into itself, so this problem does not have a symmetry in the classical sense. However, the feasible set has an obvious axis of symmetry, denoted by the dashed line. Solutions $(0,0)$ and $(1,1)$ are equivalent, with the same objective value. Our goal is to algorithmically identify symmetries of this kind. Intuitively, these are symmetries that map the 0 value of a variable to the 1 value of another variable. 

\subsection{Definition and basic properties}

Consider again problem \eqref{eq:MILPorig}. We will study symmetries of the following form:
\begin{align}
 x\mapsto Qx+q,
\end{align}
where $Q$ is a matrix with exactly one $\pm1$ coefficient in each row and column, i.e., a signed permutation matrix. The pair $(Q,q)$ is said to be a \emph{signed symmetry} or \emph{complement symmetry} of problem \eqref{eq:MILPorig} if this mapping does not change the objective value and the feasibility of a solution vector $x$. More precisely, this happens if for all $x$ we have 
\begin{align}\label{eq:sym_c_req}
 c^Tx = c^T(Qx+q),
\end{align}
and we have a regular permutation matrix $P$ such that again for all $x$ we have
\begin{align}\label{eq:sym_b_req}
b-Ax = P\left(b-A(Qx+q)\right). 
\end{align}
This means that the solution $x$ and its image $Qx+q$ are equivalent for \eqref{eq:MILPorig}.

Since conditions \eqref{eq:sym_c_req} and \eqref{eq:sym_b_req} have to hold for all $x$, we can simplify them as follows:
\begin{align}\label{eq:sym_bc_simple}
 Q^Tc &= c\\
\notag  q^Tc &= 0\\
 \notag PAQ &= A\\
 \notag P(b-Aq) &= b.
\end{align}

In addition to these requirements, our complement symmetry $(Q,q)$ has to satisfy one additional condition: it has to map binary variables to binary variables.\footnote{We have to deal with this additive term $q$ only because the range of our variables is not symmetric about the origin. If we used $\pm1$ variables instead of binaries, then we would not need it. Thus it is not surprising that $q$ is completely defined by $Q$.} Formally:
\begin{align}\label{eq:sym_bin}
 Qx+q\in\set{0,1}^n\ \text{if and only if}\ x\in\set{0,1}^n.
\end{align}
It turns out that this can be achieved by selecting $q$ appropriately:
\begin{lemma}
 If $q=\left(e-Qe\right)/2$, then $Qx+q$ is binary if and only if $x$ is binary.
\end{lemma}
\begin{proof}
Let $x$ be an arbitrary binary vector and let $i$ be an index, and let $Q_i$ be the $i$th row of $Q$. Now $Q_i$ has exactly one nonzero, let it be at the $j$th position, thus $Q_{ij} = \pm1$. The $i$th component of $Qx+q$ is then
\begin{align}
 Q_i^Tx + \frac{1 - Q_i^Te}{2} = Q_{ij}x_i + \frac{1 - Q_{ij}}{2} = \begin{cases}
                                                                     1-x_j, &\text{if}\ Q_{ij}=-1\\
                                                                     x_j, &\text{if}\ Q_{ij}=1.
                                                                    \end{cases}
\end{align}
This shows that $(Qx+q)_i$ depends only on $x_j$, and that it is binary.
\end{proof}
\begin{corollary}
Since 
\begin{align*}
q^Tc = \frac{(e-Qe)^Tc}{2} = \frac{e^T(c-Q^Tc)}{2},
\end{align*}
$Q^Tc=c$ implies $q^Tc=0$, so with this choice of $q$ we do not have to check the second condition in \eqref{eq:sym_bc_simple}, since it will follow from the first one.
\end{corollary}

\subsection{Signed permutation matrices}

Matrices with the structure of matrix $Q$ are called signed permutation matrices in the literature (see \cite{kerber71}). They are especially important for us, since they form a group, called the hyperoctahedral group, which happens to be the symmetry group of the hypercube (and by duality, the symmetry group of the hyperoctahedron, whence the name). Any symmetry of the hypercube, thus any symmetry that maps binary solutions to other binary solutions can be described by a signed permutation matrix. In this sense, signed permutations are exactly the right set of symmetries to consider for binary problems.

In comparison, regular symmetries form a much smaller group as they always leave the origin fixed.

\subsection{Interpretation}

The symmetries defined by $(Q,q)$ can map a binary variable into the complement of another (or the same) binary variable. This greatly enhances the group of symmetries that we can use for reducing the search space. In particular, if $Q_{ij}$ is 1, then $x_j$ is mapped to $x_i$ as usual, but when $Q_{ij}$ is $-1$, then $x_j$ is mapped to the complement of $x_i$. Necessarily, $x_i$ is mapped to the complement of $x_j$.

Consider for example the following problem:
\begin{align}
 \notag \min\ x_1-x_2 + x_3&\\
 \label{eq:3dexample} x_1- x_2 + x_3&\leq 1\\
 \notag x_1,x_2,x_3\ &\text{binary},
\end{align}
whose feasible set is depicted in Figure \ref{fig:3dsym}. 
\begin{figure}[ht]
\begin{center}
\begin{tikzpicture}[domain=-0.2:1.2,scale=2.8,tdplot_main_coords]
\draw[->] (0,0,0) -- (1.1,0,0) node[anchor=north east]{$x_1$};
\draw[->] (0,0,0) -- (0,1.1,0) node[anchor=north west]{$x_2$};
\draw[->] (0,0,0) -- (0,0,1.1) node[anchor=south]{$x_3$};
\draw[very thin] (0,0,1) -- (0,1,1) -- (0,1,0) -- (1,1,0) -- (1,0,0) -- (1,0,1) -- (0,0,1);
\draw[very thin] (1,1,1) -- (0,1,1);
\draw[very thin] (1,1,1) -- (1,0,1);
\draw[very thin] (1,1,1) -- (1,1,0);
\draw[very thin,dashed] (1,0,1) -- (0,1,0); 
\draw[very thick] (1,1,0) -- (1,1,1) -- (1,0,0) -- (0,0,1) -- (0,1,1) -- (0,1,0) -- (1,1,0) -- (1,0,0); 
\draw[very thick] (1,1,1) -- (0,1,1); 
\draw[very thick] (1,1,1) -- (0,0,1); 
\filldraw (.6667, .3333, .6667) circle (0.015);
\end{tikzpicture}
\end{center}
\caption{A polyhedron inside the unit cube with the symmetry $(1 \bar2 3)$.}\label{fig:3dsym}
\end{figure}
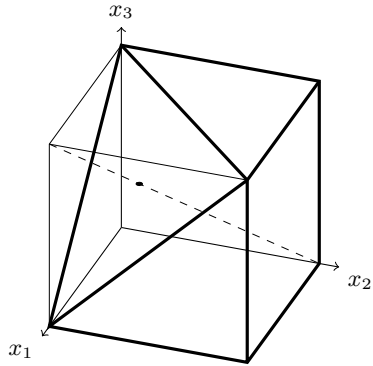
Geometrically, the symmetries are $120^\circ$ rotations about the dashed line. Extending the cycle notation for signed permutations we can write $(1 \bar2 3)$ to denote that $x_1$ is mapped to the complement of $x_2$, which in turn is mapped to $x_3$, which is mapped to $x_1$. This can be written equivalently as $(\bar1 2 \bar 3)$ using the complements of the variables.

\section{Finding complement-based symmetries}\label{sec:compsymalg}

Now that we have introduced complement symmetries we are going to present a way to obtain them from the problem description. The crux of the algorithm is a combination of a lifting procedure with the procedure of finding regular symmetries. Consider the following lifted problem (compare with \eqref{eq:MILPorig}):
\begin{align}
 \notag \min\ &(c^Tx -c^T\bar x + c^Te)/2\\
 \label{eq:sym_MILP_lifted} Ax - A\bar x &\leq 2b-Ae\\
 \notag x + \bar x &= 1\\
 \notag x, \bar x&\in\set{0,1}^n.
\end{align}
Here $\bar x$ is merely a symbol that denotes a variable. The constraints actually force it to be the complement of $x$, whence the notation, but it is important to understand that it is not the notation that makes it the complement but the constraints. The optimization problem \eqref{eq:sym_MILP_lifted} is closely related to \eqref{eq:MILPorig}:
\begin{lemma}
If $x$ is a feasible solution of \eqref{eq:MILPorig} then $(x,\bar x)$ is a feasible solution for \eqref{eq:sym_MILP_lifted} with the same objective value, and vica versa.
\end{lemma}
\begin{proof}
For the forward implication let $x$ be a feasible solution of \eqref{eq:MILPorig} and let $\bar x$ be its complement. Then we can write:
\begin{align}
\left(c^Tx - c^T\bar x + c^Te\right)/2=(c^Tx + c^Tx -c^Te + c^Te) /2 = c^Tx
\end{align}
and  
\begin{align}
Ax - A\bar x = Ax -Ae + Ax \leq b - Ae + b = 2b -Ae
\end{align}
proving that $(x,\bar x)$ is a feasible solution for \eqref{eq:sym_MILP_lifted} with the same objective value.

For the backward direction let $(x,\bar x)$ be a feasible solution of \eqref{eq:sym_MILP_lifted}. From the second constraint we get that $\bar x$ is in fact the complement of $x$, i.e., $\bar x = e-x$. Now
\begin{align}
c^Tx =c^T(x/2 - \bar x/2) + c^T(x/2 + \bar x/2) =(c^Tx -c^T\bar x +c^Te)/2
\end{align}
and 
\begin{align}
Ax = A(x/2 - \bar x/2) + A(x/2 + \bar x/2) \leq b-Ae/2 + Ae/2 = b.
\end{align}
\end{proof}
The main result of this section is the following theorem:
\begin{theorem}\label{thm:complifting}
 There is a one-to-one correspondence between the regular, permutation symmetries of problem \eqref{eq:sym_MILP_lifted} and the complement-based symmetries of problem \eqref{eq:MILPorig}.
\end{theorem}
\begin{proof}
The proof is constructive. First, let $(Q,q)$ together with $P$ be a signed symmetry of \eqref{eq:MILPorig}. Let $Q_+\geq0$ and $Q_-\leq0$ be the positive and negative parts of $Q$ such that $Q=Q_+ + Q_-$. Define the following matrices:
\begin{align}
 \widetilde P&=\begin{pmatrix}
                P & 0\\0 & \left(Q_+-Q_-\right)^{-1}
               \end{pmatrix}
\\
 \widetilde Q&=\begin{pmatrix}
                Q_+ & -Q_-\\-Q_- & Q_+
               \end{pmatrix}.
\end{align}
We claim that these two matrices together define a problem symmetry for \eqref{eq:sym_MILP_lifted}. First of all notice that $Q_+-Q_-$ is a regular permutation matrix, so its inverse exists and $\widetilde P$ is well-defined. Moreover, because of the zero blocks on the off-diagonal, $\widetilde P$ does not mix the two kinds of constraints in \eqref{eq:sym_MILP_lifted} together.

We have to check the objective, the right-hand side and the coefficient matrix of \eqref{eq:sym_MILP_lifted} under this symmetry:

The objective function is easy:
\begin{align}
{ \widetilde Q}^T \begin{pmatrix}
              c\\-c
             \end{pmatrix}&=\begin{pmatrix}
             (Q_+ + Q_-)^Tc \\ -(Q_- + Q_+)^Tc
             \end{pmatrix}=\begin{pmatrix}
             Q^Tc \\ -Q^Tc
             \end{pmatrix}=\begin{pmatrix}
              c\\-c
             \end{pmatrix},
\end{align}
where we used \eqref{eq:sym_c_req} in the last step.

The right-hand side is also very straightforward. It is obvious that we only need to check the effect of $P$, as the lower right block of $\widetilde P$ acts on an all-1 vector on the right-hand side:
\begin{align}
 \notag P(2b-Ae) &= 2Pb - PAe = 2(b + PAq)-PAe = \\
  \notag &=2b + 2PA(e - Qe)/2 - PAe = 2b + PAe - PAQe - PAe =\\
  &=2b - Ae,
\end{align}
where we used the definition of $q$ and the requirements \eqref{eq:sym_bc_simple}.
Finally, we can turn to the coefficient matrix:
\begin{align}
 \notag   \widetilde P \begin{pmatrix}
               A &-A \\ I & I
              \end{pmatrix}\widetilde Q &=
              \begin{pmatrix}
                P & 0\\0 & \left(Q_+-Q_-\right)^{-1}
               \end{pmatrix}
               \begin{pmatrix}
               A &-A \\ I & I
              \end{pmatrix}
              \begin{pmatrix}
                Q_+ & -Q_-\\-Q_- & Q_+
               \end{pmatrix} = \\
\notag          &=\begin{pmatrix}
                P & 0\\0 & \left(Q_+-Q_-\right)^{-1}
               \end{pmatrix}
               \begin{pmatrix}
               AQ_+ +AQ_-&-AQ_- - AQ_+ \\ Q_+-Q_- & -Q_- +Q_+
              \end{pmatrix}=\\
\notag           &=\begin{pmatrix}
                P & 0\\0 & \left(Q_+-Q_-\right)^{-1}
               \end{pmatrix}
               \begin{pmatrix}
               AQ &  -AQ \\ Q_+-Q_- & Q_+ -Q_-
              \end{pmatrix}=\\
\notag           &=\begin{pmatrix}
                PAQ & -PAQ\\\left(Q_+-Q_-\right)^{-1}\left(Q_+-Q_-\right) & \left(Q_+-Q_-\right)^{-1}\left(Q_+-Q_-\right)
               \end{pmatrix}=\\
              &=\begin{pmatrix}
                A & -A\\I & I
               \end{pmatrix}.              
\end{align}
This shows that $\widetilde P$ and $\widetilde Q$ are indeed symmetries of \eqref{eq:sym_MILP_lifted}.

Now let us prove the opposite direction. Take a regular permutation symmetry of \eqref{eq:sym_MILP_lifted}. Obviously, we cannot map the different kinds of constraints into each other, so the row permutation matrix will look like this:
\begin{align}
\begin{pmatrix}
 P_1 & 0 \\
 0 & P_2
\end{pmatrix},
 \end{align}
where $P_1$ and $P_2$ are permutation matrices acting on the two kinds of constraints. The column permutations however can be quite complex, so for now let us just partition it according to the block structure of the problem:
\begin{align}
 \begin{pmatrix}
  Q_1 & Q_2 \\ Q_3 & Q_4
 \end{pmatrix}.
\end{align}
These matrices satisfy the following:
\begin{align}
 \label{eq:sym_lifter_cond_A}\begin{pmatrix}
 P_1 & 0 \\
 0 & P_2
\end{pmatrix}
\begin{pmatrix}
                A & -A\\I & I
               \end{pmatrix}\begin{pmatrix}
  Q_1 & Q_2 \\ Q_3 & Q_4
 \end{pmatrix} &= \begin{pmatrix}
                A & -A\\I & I
               \end{pmatrix}\\
\label{eq:sym_lifter_cond_c}\begin{pmatrix}
  Q_1 & Q_2 \\ Q_3 & Q_4
 \end{pmatrix}               
\begin{pmatrix}
              c\\-c
             \end{pmatrix} & = \begin{pmatrix}
              c\\-c
             \end{pmatrix}\\
\label{eq:sym_lifter_cond_b}    \begin{pmatrix}
 P_1 & 0 \\
 0 & P_2
\end{pmatrix}\begin{pmatrix}
              2b - Ae\\ e
             \end{pmatrix} &= \begin{pmatrix}
              2b - Ae\\ e
             \end{pmatrix}
\end{align}             
Now working out the lower two terms of \eqref{eq:sym_lifter_cond_A} we get 
\begin{align*}
 P_2(Q_1+Q_3)&=I\\
 P_2(Q_2+Q_4)&=I,
\end{align*}
which implies $Q_1+Q_3=Q_2+Q_4= P_2^{-1}$. We are now going to prove that $Q_1=Q_4$ and $Q_2=Q_3$. For this let us consider the case when ${Q_1}_{ij}=1$ for some indices $i$ and $j$. Since the entire matrix is a permutation matrix, this implies that ${Q_2}_{ik}=0$ and ${Q_3}_{kj}=0$ for all values of $k$. Using $Q_1+Q_3=Q_2+Q_4$ we get that
\begin{align}
{Q_4}_{ij} = \left(Q_2+Q_4\right)_{ij} =  \left(Q_1+Q_3\right)_{ij} = 1.
\end{align}
Swapping the role of $Q_1$ and $Q_4$ we get that ${Q_4}_{ij}=1$ implies ${Q_1}_{ij}=1$, so a coefficient in $Q_1$ is 1 if and only if it is also 1 in $Q_4$. This proves that $Q_1=Q_4$. The $Q_2=Q_3$ result is proved similarly.

Now we can unambiguously define 
\begin{align}
 Q&:=Q_1-Q_2 = Q_4-Q_3,\ \text{and}\\
 q&:=\frac{e-Qe}{2}
\end{align}
Now we claim that $(Q,q)$ together with $P_1$ is a signed symmetry for \eqref{eq:MILPorig}, for which we only need to check the three conditions in \eqref{eq:sym_bc_simple}.

The objective:
\begin{align*}
 Q^Tc = (Q_1-Q_2)^Tc = c
\end{align*}
because of \eqref{eq:sym_lifter_cond_c}. 
The coefficient matrix (using $Q=(Q_1 -Q_2+Q_4 - Q_3)/2$)
\begin{align*}
 P_1AQ &= P_1A(Q_1 -Q_2+Q_4 - Q_3)/2 =\\
 &=P_1A(Q_1-Q_3)/2 - P_1A(Q_2-Q_4)/2 = A/2 - (-A)/2 = A,
\end{align*}
where we used the upper part of \eqref{eq:sym_lifter_cond_A}. Finally, the right-hand side:
\begin{align*}
 P_1(b-Aq) &= P_1b - P_1A(e - Qe)/2 = P_1(2b-Ae)/2 + P_1AQe/2 =\\
 &=(2b-Ae)/2 + Ae/2 = b,
\end{align*}
where we used the upper part of \eqref{eq:sym_lifter_cond_b}.

Notice that the correspondence is indeed a bijection, as all the possible permutation symmetries of \eqref{eq:sym_MILP_lifted} must have the structure of $\widetilde P$ and $\widetilde Q$. 

This completes the proof of Theorem \ref{thm:complifting}.
\end{proof}
\begin{remark}
The additional constraints in \eqref{eq:sym_MILP_lifted} ensure that the complement of the image is the image of the complement. They are necessary and cannot be dropped, otherwise there is nothing to link a variable to its complement. This is shown by the fact that they were used to prove that a permutation symmetry of \eqref{eq:sym_MILP_lifted} has the special block structure.
\end{remark}

\subsection{Implementation details}

As the lifted problem \eqref{eq:sym_MILP_lifted} uses only data that is present in the original problem \eqref{eq:MILPorig} we can modify the symmetry detection algorithms (see, e.g., \cite{katebi12}) to search for complement symmetries as well. This will incur only a moderate storage overhead, but obviously, the coefficient matrix would not have to be stored in its lifted form, the algorithm can be custom tailored to operate only on the original data. 

\subsection{Example}

Consider again problem \eqref{eq:3dexample} depicted in Figure \ref{fig:3dsym}. The lifted version is:
\begin{align}
 \notag \min\ (x_1-x_2 + x_3 - \bar x_1 + \bar x_2 - \bar x_3 + 1)/2&\\
 \label{eq:3dexamplelifted} x_1- x_2 + x_3 - \bar x_1 + \bar x_2 - \bar x_3&\leq 0\\
\notag x_i + \bar x_i &=1, \, i=1,2,3\\
 \notag x_1,x_2,x_3,\bar x_1, \bar x_2, \bar x_3\ &\text{binary}.
\end{align}
Running symmetry detection on this problem we get that $x_1$, $\bar x_2$ and $x_3$ belong to the same orbit, exactly as we expected by looking at Figure \ref{fig:3dsym}. The symmetries for problem \eqref{eq:3dexamplelifted} are $(1\bar2 3)$ and, equivalently, $(\bar1 2\bar3)$.

\subsection{Using signed symmetries in heuristics}\label{sec:compheur}

Signed symmetries can be used in the regular way in orbital branching and orbital fixing, see \cite{ostrowski11} for details. However, a probably more useful way to use them is to apply them to feasible solutions. Signed permutations can map a solution into a broader set of solutions by being able to change the number of nonzeros. Also, by modifying problem \eqref{eq:MILPsym} slightly, we can use signed symmetries for orbit MIPping. Again, let $\tilde x$ be a feasible solution and let $\cO_1, \dots, \cO_k$ be the signed orbit partition of problem \eqref{eq:MILPorig}.
\begin{align}
 \notag \min\ &c^Tx\\
 \label{eq:MILPsym}Ax&\leq b\\
\notag \sum_{i\in\cO_j}x_i + \sum_{\bar i\in\cO_j}(1 - x_i) & = \sum_{i\in\cO_j}\tilde x_i + \sum_{\bar i\in\cO_j}(1-\tilde x_i),\, \forall j = 1,\dots,k\\
 \notag x&\in\set{0,1}^n, 
\end{align}
where $\bar i\in\cO_j$ denotes the relation that the complement of variable $\tilde x_i$ belongs to the orbit $\cO_j$.

\section{Extensions and future work}

The developments in this paper can be naturally extended to problems containing other types of variables. We chose a pure binary problem only to simplify the notation. Also, with a slight modification to the lifted problem \eqref{eq:sym_MILP_lifted} it is possible to map the lower bound of any boxed variable to the upper bound of a boxed variable. Looking at the conditions in \eqref{eq:sym_bc_simple} we have showed that signed symmetries are the largest class of symmetries that also map binary variables to binary variables. However, these conditions can be satisfied by other symmetries, which are thus not symmetries of the hypercube. These can be useful for special problem classes. These symmetries are the subject of future research.

\bibliographystyle{splncs}
\bibliography{SAS_symmetry_2014}
\end{document}